\documentclass[12pt]{amsart}

\usepackage{fullpage}
\usepackage{amsmath}
\usepackage{amssymb}
\usepackage{amsthm}
\usepackage{amstext}
\usepackage{appendix}
\usepackage{perpage}
\usepackage{color}
\usepackage{tikz-cd}
\usepackage[colorlinks=true,linkcolor=blue, citecolor=blue]{hyperref}

\newtheorem{thm}{Theorem}[section]
\newtheorem{lemm}[thm]{Lemma}

\theoremstyle{definition}

\newtheorem{defi}[thm]{Definition}

\begin{document}

\title{On the inner products of some Deligne--Lusztig type representations}

\author{Zhe Chen}

\address{Department of Mathematics, Shantou University, Shantou, China}

\email{zhechencz@gmail.com}

\begin{abstract}
In this paper we introduce a family of Deligne--Lusztig type varieties attached to connected reductive groups over quotients of discrete valuation rings, naturally generalising the higher Deligne--Lusztig varieties and some constructions related to the algebraisation problem raised by Lusztig. We establish the inner product formula between the representations associated to these varieties and the higher Deligne--Lusztig representations.
\end{abstract}

\maketitle

\tableofcontents

\section{Introduction}\label{sec:Intro}

Let $\mathcal{O}$ be a complete discrete valuation ring with a uniformiser $\pi$ and with a finite residue field $\mathcal{O}/\pi=\mathbb{F}_q$. Since $\mathcal{O}$ is a profinite ring, the study of smooth representations of reductive groups over $\mathcal{O}$ leads to the study of representations of the groups over $\mathcal{O}/\pi^r$, for all $r\in\mathbb{Z}_{>0}$; in this paper, we will be concerned with these latter objects.

\vspace{2mm} Let $\mathbb{G}$ be a connected reductive group scheme over $\mathcal{O}_r:=\mathcal{O}/\pi^r$. When $r=1$, $\mathbb{G}(\mathcal{O}_1)=\mathbb{G}(\mathbb{F}_q)$ is a finite group of Lie type; Deligne and Lusztig \cite{DL1976} founded a geometric approach to its representations. In the Corvallis paper \cite{Lusztig1979SomeRemarks}, Lusztig proposed a generalisation of this geometric theory for $\mathbb{G}(\mathcal{O}_r)$, for any positive integer $r\geq1$. The proofs in the positive characteristic case was presented by Lusztig himself in \cite{Lusztig2004RepsFinRings}, which was later generalised by Stasinski for the general case in \cite{Sta2009Unramified}, by the use of the Greenberg functor. When the involved parameters satisfying some regularity conditions, these representations are irreducible; meanwhile, for $r\geq2$, following Shintani \cite{Shintani1968sqr_int_irrep_lin_gr}, G\'erardin \cite{Gerardin1975SeriesDiscretes} found a purely algebraic method to construct some irreducible representations of these groups.

\vspace{2mm} It is a very interesting question that whether the geometrically constructed representations of Lusztig coincides with the algebraically constructed representations of G\'erardin. At even levels (i.e.\ $r$ is even), this problem was investigated recently, and was proved to be true, as expected by Lusztig; for $\mathbb{G}=\mathrm{GL}_n$ this is in \cite{ZheChen_PhDthesis} and for a general $\mathbb{G}$ this is in \cite{ChenStasinski_2016_algebraisation}.

\vspace{2mm} In this paper, we introduce a family of varieties $L^{-1}(FU^{r-b,b})$, labelled by $b\in [0,r]\cap\mathbb{Z}$, such that the alternating sum (see Definition~\ref{defi: DL rep at various pages})
\begin{equation}\label{key object}
R_{T,U,b}^{\theta}:=\sum_{i\in\mathbb{Z}}(-1)^{i}H^{i}_c(L^{-1}(FU^{r-b,b}),\overline{\mathbb{Q}}_{\ell})_{\theta}
\end{equation}
is a virtual representation of $\mathbb{G}(\mathcal{O}_r)$, where the parameter $\theta$ is a character of certain finite abelian group $T^F$ (here $\ell$ is a prime not equal to $\mathrm{char}(\mathbb{F}_q)$). When $b=r$, this coincides with the higher Deligne--Lusztig representation $R_{T,U}^{\theta}$ introduced in \cite{Lusztig1979SomeRemarks}; when $r$ is even and $b=r/2$, this coincides with G\'{e}rardin's representation introduced in \cite{Gerardin1975SeriesDiscretes}. 

\vspace{2mm} Our aim is to compute the inner product of the higher Deligne--Lusztig representation $R_{T,U}^{\theta}$ and the general $R_{T,U,b}^{\theta}$: In particular, we will show that
\begin{equation}\label{main result}
\langle R_{T,U,b}^{\theta}, R_{T,U}^{\theta} \rangle_{\mathbb{G}(\mathcal{O}_r)}=1,
\end{equation}
when $\theta$ is regular and in general position (see Theorem~\ref{main theorem}).

\vspace{2mm} Computing inner products of this type is one of the core steps in many situations in Deligne--Lusztig theory. Here the general principle is that one should first transfer the inner product of this type into an (equivariant) Euler characteristic of certain variety (see Lemma~\ref{lemma: transfer into Euler characteristic}), then partition the variety into small pieces, which reduces the problem of computing the Euler characteristic of the variety to computing that of each small pieces; however, to make this argument practically work, one usually needs more sophisticated constructions to deal with different specific situations.

\vspace{2mm} Note that, in the case $b=r$, \eqref{main result} implies the irreducibility of the higher Deligne--Lusztig representations, and was proved in \cite{Lusztig2004RepsFinRings} with $\mathrm{char}(\mathcal{O})>0$ and in \cite{Sta2009Unramified} in general. 

\vspace{2mm} Also note that, in the case $r$ is even and $b=r/2$, \eqref{main result}  is the main step in the algebraisation of higher Deligne--Lusztig representations at even levels, and was proved in \cite{ZheChen_PhDthesis} with $\mathbb{G}=\mathrm{GL}_n$ and in \cite{ChenStasinski_2016_algebraisation} in general.

\vspace{2mm} However, the methods in the $b=r$ case and the $b=r/2$ case faced some obstructions in the general case; the reasons are roughly that the variety $U^{r-b,b}$ is in general not the Greenberg functor image of the unipotent radical of a Borel subgroup (while $U^{0,r}=U$ is so), and in general not stable under the Frobenius action or the Weyl group conjugation (while $U^{r/2,r/2}$ is so). In this paper we will overcome these obstructions, thus complete \eqref{main result} for any $r$ and $b$. In the special case that $r$ is odd and $b=(r+1)/2$, this result is expected to be useful in the algebraisation of higher Deligne--Lusztig representations at odd levels.

\vspace{2mm} For $x,y$ in an algebraic group, we will use the conjugation notation $x^y=y^{-1}xy={^{y^{-1}}x}$. For the alternating sum of $\ell$-adic cohomology groups $\sum_i(-1)^iH^i_c(-,\overline{\mathbb{Q}}_{\ell})$, we will write $H^*_c(-,\overline{\mathbb{Q}}_{\ell})$ for short. If an involved Frobenius $F$ is applied to some object $*$, then we occasionally drop the ``$()$'' in $F(*)$ and write $F*$.

\vspace{2mm}\noindent {\bf Acknowledgement.} The author thanks Alexander Stasinski for helpful discussions. During the preparation of this work, the author is partially supported by the STU Scientific Research Foundation for Talents NTF17021.

\section{Deligne--Lusztig varieties at various pages}\label{section: DL at various pages}

Fix an arbitrary $r\in\mathbb{Z}_{>0}$. Let $\mathcal{O}^{\mathrm{ur}}$ be the ring of integers in the maximal unramified extension of the fraction field of $\mathcal{O}$, and put $\mathcal{O}^{\mathrm{ur}}_r:=\mathcal{O}^{\mathrm{ur}}/\pi^r$. For a smooth affine group scheme $\mathbf{H}$ over $\mathcal{O}^{\mathrm{ur}}_r$, according to the Greenberg functor $\mathcal{F}$ introduced in \cite{Greenberg19611} and \cite{Greenberg19632}, there is an associated affine algebraic group $H=\mathcal{F}\mathbf{H}$ over $k:=\overline{\mathbb{F}}_q$ satisfying several nice properties. In \cite{Sta2009Unramified}, this Greenberg functor technique was used to generalise the constructions and results in \cite{Lusztig2004RepsFinRings}, from the positive characteristic case to the general case. Detailed modern treatments of the Greenberg functors can be found in \cite{Sta2012ReductiveGr} and \cite{Berthapelle_Gonzalez_GreenbergRevisited}.

\vspace{2mm} Let $\mathbb{G}$ be a connected reductive group over $\mathcal{O}_r$ (i.e.\ a smooth affine group scheme over $\mathrm{Spec}(\mathcal{O}_r)$, whose geometric fibres are connected reductive algebraic groups in the usual sense; this is the definition used in \cite[XIX 2.7]{SGA3}), and let $\mathbf{G}$ be the base change of $\mathbb{G}$ to $\mathcal{O}^{\mathrm{ur}}_r$. In general, a surjective algebraic group endomorphism on  $G=G_r=\mathcal{F}\mathbf{G}$ with finitely many fixed points is called a Frobenius endomorphism. In this paper we only focus on the following typical situation: The Frobenius element in $\mathrm{Gal}(k/\mathbb{F}_q)$ extends to an automorphism of $\mathcal{O}^{\mathrm{ur}}_r$, then by the Greenberg functor it gives a rational structure of $G$ over $\mathbb{F}_q$, such that $\mathbf{G}(\mathcal{O}^{\mathrm{ur}}_r)\cong G(k)$ and $\mathbb{G}(\mathcal{O}_r)\cong G^F$ as abstract groups, where $F$ denotes the associated geometric Frobenius endomorphism (see the terminology in \cite[Chapter 3]{DM1991}); this allows us to use the geometry of $G$ to study the representations of $\mathbb{G}(\mathcal{O}_r)$. Let $L\colon g\mapsto g^{-1}F(g)$ be the Lang endomorphism on $G$.

\vspace{2mm} To define our fundamental objects \eqref{key object}, we recall some notation used in \cite{Sta2009Unramified} and \cite{ChenStasinski_2016_algebraisation}: For any $i\in[1,r]\cap\mathbb{Z}$, the modulo $\pi^{i}$ reduction map gives a surjective algebraic group endomorphism $G_r\rightarrow G_i$. We denote the kernel subgroup by $G^{i}=G_r^{i}$, and put $G^0:=G$ (do not mix it with the identity component $G^{\circ}$); similar notation applies to the closed subgroups of $G$. Given a closed subgroup $H$ of $G$, we call it $F$-stable (or $F$-rational, or simply rational) if $F(H)\subseteq H$. Let $\mathbf{T}$ be a maximal torus of $\mathbf{G}$ such that $T=\mathcal{F}\mathbf{T}$ is $F$-stable, let $\mathbf{B}$ be a Borel subgroup containing $\mathbf{T}$, and let $\mathbf{U}$ (resp.\ $\mathbf{U}^-$) be the unipotent radical of $\mathbf{B}$ (resp.\ the opposite of $\mathbf{B}$). Denote the associated algebraic groups by $B=\mathcal{F}\mathbf{B}$, $U=\mathcal{F}\mathbf{U}$, and $U^-=\mathcal{F}\mathbf{U}^{-}$, respectively.

\vspace{2mm} For any $b\in[0,r]\cap\mathbb{Z}$, consider the unipotent algebraic group $U^{r-b,b}:=U^{r-b}(U^-)^{b}$. Note that, when $r$ is even and $b=r/2$, this is a commutative group, and in this case it is denoted by $U^{\pm}$ in \cite{ChenStasinski_2016_algebraisation}.

\begin{defi}\label{defi: DL var at various pages}
For $b\in [0,r]\cap\mathbb{Z}$, we call $L^{-1}(FU^{r-b,b})\subseteq G$ a Deligne--Lusztig variety at page $b$.
\end{defi}

Note that $G^F$ acts on $L^{-1}(FU^{r-b,b})$ by the left multiplication, and $T^F$ acts on $L^{-1}(FU^{r-b,b})$ by the right multiplication, so, after fixing an arbitrary rational prime $\ell\nmid q$, we obtain the following construction.

\begin{defi}\label{defi: DL rep at various pages}
For $b\in [0,r]\cap\mathbb{Z}$ and $\theta\in\widehat{T^F}:=\mathrm{Hom}(T^F,\overline{\mathbb{Q}}_{\ell})$, we call the virtual $G^F$-representation 
\begin{equation*}
R_{T,U,b}^{\theta}:=H^*_c(L^{-1}(FU^{r-b,b}),\overline{\mathbb{Q}}_{\ell})_{\theta}
\end{equation*} 
a Deligne--Lusztig representation at page $b$.
\end{defi}

This generalises the constructions studied previously in \cite{Lusztig1979SomeRemarks}, \cite{Lusztig2004RepsFinRings}, \cite{Sta2009Unramified}, \cite{ZheChen_PhDthesis}, and \cite{ChenStasinski_2016_algebraisation}. In those works one central theme is the inner product formula; in this paper we continue this theme by completing the formula between the higher Deligne--Lusztig representation $R_{T,U}^{\theta}=R_{T,U,0}^{\theta}$ and the general $R_{T,U,b}^{\theta}$, for any $b$. To state our main result, we need to recall the notion of regularity:

\vspace{2mm} Let $\Phi=\Phi(\mathbf{G},\mathbf{T})$ be the set of roots of $\mathbf{T}$, let $\Phi^+\subseteq \Phi$ be the subset of positive roots with respect to $\mathbf{B}$, and let $\Phi^-:=\Phi\backslash\Phi^+$ be the corresponding subset of negative roots. Given a root $\alpha\in\Phi$, we denote by $\mathbf{T}^{\alpha}$ the image of the coroot $\check{\alpha}$, and write $T^{\alpha}:=\mathcal{F}\mathbf{T}^{\alpha}$; similarly, we write $\mathbf{U}_{\alpha}\subseteq\mathbf{U}$ for the root subgroup of $\alpha$, and write $U_{\alpha}$ for the Greenberg functor image. Put $\mathcal{T}^{\alpha}:=(T^{\alpha})^{r-1}$.

\begin{defi}
Let $a$ be a positive integer such that $\mathcal{T}^{\alpha}$ is $F^{a}$-stable for all $\alpha\in\Phi$. Consider the norm map $N^{F^{a}}_F(t):=t\cdot F(t)\cdots F^{a-1}(t)$ on $T^{F^{a}}$. A character $\theta\in\widehat{T^F}$ is called regular, if it is non-trivial on $N^{F^{a}}_F(\mathcal{T}^{\alpha})$ for every $\alpha\in\Phi$.
\end{defi}

The notion of regularity is indeed independent of the choice of the integer $a$; see \cite[1.5]{Lusztig2004RepsFinRings} and \cite[2.8]{Sta2009Unramified}.

\vspace{2mm} Since $\mathcal{O}^{\mathrm{ur}}_r$ is a strictly henselian local ring, by \cite[2.1]{Sta2009Unramified} and \cite[XXII 3.4]{SGA3} we see that $W(T):=N_G(T)/T\cong W(T_1):=N_{G_1}(T_1)/T_1$. Now we can state the formula:
\begin{thm}\label{main theorem}
Let $b\in[0,r]\cap \mathbb{Z}$. If $\theta\in \widehat{T^F}$ is regular, then 
\begin{equation*}
\langle R_{T,U,b}^{\theta},R_{T,U}^{\theta} \rangle_{\mathbb{G}(\mathcal{O}_r)}=\#\mathrm{Stab}_{W(T)^F}(\theta).
\end{equation*}
In particular, if $\theta$ is moreover in general position, then $\langle R_{T,U,b}^{\theta},R_{T,U}^{\theta} \rangle_{\mathbb{G}(\mathcal{O}_r)}=1$.
\end{thm}
The remaining part of this paper devotes to its proof.

\section{Proof of Theorem~\ref{main theorem}}\label{section: proof}

We start with the following general lemma which allows us to transform the inner product into a Euler characteristic.

\begin{lemm}\label{lemma: transfer into Euler characteristic}
Suppose that we are given a connected algebraic group over ${\mathbb{F}}_q$; denote by $\mathcal{G}$ its base change to $\overline{\mathbb{F}}_q$. Let $\sigma$ be the associated geometric Frobenius endomorphism, and let $L_{\sigma}$ be the associated Lang morphism. If $X$ and $X'$ are two closed subvarieties of $\mathcal{G}$, then the morphism
\begin{equation*}
\kappa\colon\mathcal{G}^{\sigma} \backslash (L^{-1}(X)\times L^{-1}(X'))\longrightarrow \{ (x,x',y)\in X\times X'\times\mathcal{G} \mid  x\sigma(y)=yx'  \},
\end{equation*}
given by $(g,g')\mapsto (L_{\sigma}(g), L_{\sigma}(g'), g^{-1}g')$, is an isomorphism. (Indeed, in the proof of our main result, only the bijectivity part of this morphism is needed.)
\end{lemm}

This can be proved in the same way as the special case in \cite[Page 222]{Carter1993FiGrLieTy}:

\begin{proof}
Consider the variety $S:=\{(x,x',y)\in\mathcal{G}\times \mathcal{G}\times\mathcal{G}\mid x\sigma(y)=yx' \}$, which contains $\{ (x,x',y)\in X\times X'\times\mathcal{G} \mid  x\sigma(y)=yx'  \}$ as a closed subvariety. Note that, in the defining equation of $S$, the component $x'$ is determined by $x$ and $y$; this fact, together with the surjectivity of the Lang map, imply that the following morphism is surjective:
\begin{equation*}
\iota\colon \mathcal{G}\times\mathcal{G}\longrightarrow S,\quad (g,g')\mapsto (L_{\sigma}(g), L_{\sigma}(g'), g^{-1}g').
\end{equation*}
Note that the fibres of $\iota$ are $\mathcal{G}^{\sigma}$-orbits, so it suffices to show that the induced bijective morphism $\bar{\iota}\colon \mathcal{G}^{\sigma}\backslash (\mathcal{G}\times\mathcal{G})\rightarrow S$, extending $\kappa$, is an isomorphism. By \cite[6.6]{Borel_1991_LinearAlgGp} it suffices to show that $S$ is smooth and $\iota$ is separable.

\vspace{2mm} The smoothness of $S$ follows from the fact that $S\cong\mathcal{G}\times\mathcal{G}$ via $(x,x',y)\mapsto (x,y)$.

\vspace{2mm} Meanwhile, consider $\iota'\colon S\rightarrow \mathcal{G}\times\mathcal{G}$ given by $(x,x',y)\mapsto (x,x')$. Then $\iota'\circ\iota$ is a Lang morphism, which is known to be separable, thus $\iota$ is also separable.
\end{proof}

Now we turn to the proof of the theorem itself.

\vspace{2mm} First, when $\theta$ is regular, it is known that $R_{T,U}^{\theta}$ is independent of $U$, so we only need to deal with the case $b\in [r/2,r]\cap\mathbb{Z}$. Moreover, the case $b=r$ and the case $b=r/2$ (for $r$ even) are also known. (See \cite[Corollary 2.4]{Lusztig2004RepsFinRings}, \cite[Corollary 3.4]{Sta2009Unramified}, \cite[Theorem 4.3.2]{ZheChen_PhDthesis}, and \cite[Theorem 4.1]{ChenStasinski_2016_algebraisation}.) Thus, in order to simplify some boundary cases discussions, we will assume $b\in(r/2,r)\cap\mathbb{Z}$ till the end of this paper (note that, this is a non-empty condition only when $r\geq3$).

\vspace{2mm} By the K\"{u}nneth formula and Lemma~\ref{lemma: transfer into Euler characteristic}, we have 
\begin{equation*}
\langle R_{T,U,b}^{\theta},R_{T,U}^{\theta} \rangle_{\mathbb{G}(\mathcal{O}_r)}=\dim \sum_{i} (-1)^iH_c^i(\Sigma,\overline{\mathbb{Q}}_{\ell})_{\theta^{-1},\theta},
\end{equation*}
where
\begin{equation*}
\Sigma:=\{(x,x',y)\in F(U^{r-b,b})\times F(U)\times G\mid xF(y)=yx'  \},
\end{equation*}
on which $T^F\times T^F$ acts by $(t,t')\colon (x,x',y)\mapsto (x^t,(x')^{t'}, t^{-1}yt')$.

\vspace{2mm} We want to compute the cohomology following a general principle of Lusztig, namely, we will first decompose $\Sigma$ into locally closed pieces according to the Bruhat decomposition, and then compute the cohomology of each piece.

\vspace{2mm} Consider the Bruhat decomposition $G_1=\coprod_{v\in W(T)} B_{1}\hat{v}B_{1}$ of $G_1=\mathbf{G}(k)$ (here $\hat{v}$ denotes a lift of $v$ in $G_1$), which lifts to a decomposition $G=\coprod_{v\in W(T)}G_v$, where 
\begin{equation*}
G_v:=(U\cap\hat{v}U^{-}\hat{v}^{-1})(\hat{v}(U^{-})^1\hat{v}^{-1})\hat{v}TU
\end{equation*}
(here we again use $\hat{v}$ to denote a lift of $v$ in $G$); see e.g.\ the proof of \cite[Lemma 2.3]{Sta2009Unramified}. This results a finite partition of $\Sigma$ into locally closed subvarieties
\begin{equation*}
\Sigma=\coprod_{v\in W(T)} \Sigma_{v},
\end{equation*}
where 
\begin{equation*}
\Sigma_{v}:=\{(x,x',y)\in F(U^{r-b,b})\times FU\times G_v\mid xF(y)=yx'\}.
\end{equation*}
For each $v$, consider the variety
\begin{equation*}
\mathcal{Z}_v:=(U\cap \hat{v}U^{-}\hat{v}^{-1})\times 
\hat{v}(U^{-})^{1}\hat{v}^{-1}.
\end{equation*}
Taking the multiplication morphism $\mathcal{Z}_v\rightarrow \hat{v}U^-\hat{v}^{-1}$, we obtain a locally trivial fibration $\widehat{\Sigma}_v\rightarrow\Sigma_{v}$ by an affine space, where
\begin{equation*}
\widehat{\Sigma}_v:=\{(x,x',u',u^{-},\tau,u)\in F(U^{r-b,b})\times FU\times \mathcal{Z}_v \times T\times U\mid xF(u'u^{-}\hat{v}\tau u)=u'u^{-}\hat{v}\tau ux'\},
\end{equation*}
on which $T^F\times T^F$ acts by
\begin{equation*}
(t,t')\colon (x,x',u',u^{-},\tau,u)\longmapsto (x^t,(x')^{t'},(u')^t,(u^{-})^t,(t^{\hat{v}})^{-1}\tau t',u^{t'}).
\end{equation*}
Taking the change of variable $x'F(u)^{-1}\mapsto x'$, we can rewrite $\widehat{\Sigma}_v$ as
\begin{equation*}
\widehat{\Sigma}_v=\{(x,x',u',u^{-},\tau,u)\in F(U^{r-b,b})\times FU\times \mathcal{Z}_v\times T\times U\mid xF(u'u^{-}\hat{v}\tau)=u'u^{-}\hat{v}\tau ux'\},
\end{equation*}
on which the $T^F\times T^F$-action does not change; for our purpose it suffices to compute the dimension of the $\theta^{-1}\times\theta$-isotypical part of $H_c^*(\widehat{\Sigma}_v)$, for every $v$.

\vspace{2mm} We want to stratify each $\widehat{\Sigma}_v$ by stratifying each $\mathcal{Z}_v$. Let us first make the following notation convention. For $\beta\in\Phi^{-}$, let $F(\beta)\in\Phi$ be the root defined by $F(U)_{F(\beta)}=F(U_{\beta})$. Similarly, we can define $F$ on $\Phi^{+}$, thus obtain a bijection on $\Phi=F(\Phi^{-})\sqcup F(\Phi^{+})$ and hence also on $\{U_{\beta}\}_{\beta\in{\Phi}}$. Consider the length function on the roots with respect to $(T,FB)$; we denote by $\mathrm{ht}(-)$ the absolute value of the length function; we fix an arbitrary total order on $F(\Phi^-)$ refining the order given by $\mathrm{ht}(-)$. For $z\in U^{-}$ and $\beta'\in F(\Phi^{-})$, let $x^{F(z)}_{\beta'}\in (FU)_{\beta'}$ be defined by the decomposition 
\begin{equation}\label{formula: product wrt roots}
F(z)=\prod_{\beta\in F(\Phi^{-})}x^{F(z)}_{\beta},
\end{equation}
where the product is taken with respect to the following order: If $\beta_1<\beta_2$, then $x^{F(z)}_{\beta_1}$ is left to $x^{F(z)}_{\beta_2}$.

\vspace{2mm} For $\beta\in\Phi^{-}$, let $\mathcal{Z}_v^{\beta}\subseteq\mathcal{Z}_v$ be the subvariety consisting of the elements $(u',u^{-})$ satisfying that, for $z:=(u'u^{-})^{\hat{v}}\in U^-$, one has $x^{F(z)}_{\beta'}=1$ whenever $\beta'< F(\beta)$, and $x^{F(z)}_{F(\beta)}\neq 1$. This gives a finite stratification into locally closed subvarieties
\begin{equation*}
\mathcal{Z}_v=(\sqcup_{\beta\in\Phi^{-}}\mathcal{Z}_v^{\beta})\sqcup \mathcal{Z}_v^c,
\end{equation*}
where $\mathcal{Z}_v^c:=\mathcal{Z}_v\backslash (\sqcup_{\beta\in\Phi^{-}}\mathcal{Z}_v^{\beta})$. Furthermore, for $i=0,\cdots,r$, let $\mathcal{Z}_v^{\beta}(i)$ be the pre-image of ${^{\hat{v}}(U^{-})}\cap G^i$ (recall that $G^0:=G$) along the multiplication morphism $\mathcal{Z}_{v}^{\beta}\rightarrow {^{\hat{v}}(U^{-})}$; for $i=0,\cdots,r-1$, put $\mathcal{Z}_v^{\beta}(i)^*:=\mathcal{Z}_v^{\beta}(i)\backslash \mathcal{Z}_v^{\beta}(i+1)$. The above stratification can then be refined to be
\begin{equation*}
\mathcal{Z}_v=(\sqcup_{\beta\in\Phi^{-}}\sqcup_{i=0}^{r-1}\mathcal{Z}_v^{\beta}(i)^*)\sqcup \mathcal{Z}_v^c.
\end{equation*}
This naturally corresponds to a stratification of $\widehat{\Sigma}_v$:
\begin{equation*}
\widehat{\Sigma}_v=(\sqcup_{\beta,i}\Sigma_{v,\beta,i})\sqcup\Sigma_{v,c},
\end{equation*}
where
\begin{equation*}
\Sigma_{v,\beta,i}:=\{(x,x',u',u^{-},\tau,u)\in F(U^{r-b,b})\times FU\times \mathcal{Z}^{\beta}_v(i)^*\times T\times U\mid xF(u'u^{-}\hat{v}\tau)=u'u^{-}\hat{v}\tau ux'\}
\end{equation*}
and
\begin{equation*}
\Sigma_{v,c}:=\{(x,x',u',u^{-},\tau,u)\in F(U^{r-b,b})\times FU\times \mathcal{Z}_v^c\times T\times U\mid xF(u'u^{-}\hat{v}\tau)=u'u^{-}\hat{v}\tau ux'\}.
\end{equation*}
Let $\mathcal{Z}_v^{\beta}(i)'\subseteq\mathcal{Z}_v^{\beta}(i)^*$ be the subvariety consisting of the elements $(u',u^-)$ with $u'u^{-}\notin U^{r-b,b}$, and let $\mathcal{Z}_v^{\beta}(i)''\subseteq \mathcal{Z}_v^{\beta}(i)^*$ be the complement to $\mathcal{Z}_v^{\beta}(i)'$. So we have 
\begin{equation*}
\Sigma_{v,\beta,i}=\Sigma_{v,\beta,i}'\sqcup \Sigma_{v,\beta,i}'',
\end{equation*}
where
\begin{equation*}
\Sigma_{v,\beta,i}':=\{(x,x',u',u^{-},\tau,u)\in F(U^{r-b,b})\times FU\times \mathcal{Z}^{\beta}_v(i)'\times T\times U\mid xF(u'u^{-}\hat{v}\tau)=u'u^{-}\hat{v}\tau ux'\}
\end{equation*}
and
\begin{equation*}
\Sigma_{v,\beta,i}'':=\{(x,x',u',u^{-},\tau,u)\in F(U^{r-b,b})\times FU\times \mathcal{Z}^{\beta}_v(i)''\times T\times U\mid xF(u'u^{-}\hat{v}\tau)=u'u^{-}\hat{v}\tau ux'\}.
\end{equation*}
Note that the varieties $\Sigma_{v,\beta,i}'$, $\Sigma_{v,\beta,i}''$, and $\Sigma_{v,c}$ all inherit the $T^F\times T^F$-action in a natural way; we will compute their equivariant Euler characteristics:

\begin{lemm}\label{lemm a} One has
$\dim H^*_c(\Sigma_{v,c})_{\theta^{-1},\theta} =
\begin{cases}
1, & {\text{if\ }} v\in \mathrm{Stab}_{W(T)^F}(\theta)\\
0,  & \text{otherwise}. 
\end{cases}$
\end{lemm}

\begin{lemm}\label{lemm b}
One has $\dim H^*_c(\Sigma_{v,\beta,i}'')_{\theta^{-1},\theta} =0$, where $i=0,\cdots, r-1$ and $\beta\in\Phi^-$.
\end{lemm}

\begin{lemm}\label{lemm c}
One has $\dim H^*_c(\Sigma_{v,\beta,i}')_{\theta^{-1},\theta} =0$, where $i=0,\cdots, r-1$ and $\beta\in\Phi^-$.
\end{lemm}

(As we will see in the below, the first two lemmas are true for any $\theta$, no matter whether $\theta$ is regular; however, our proof of the third lemma relies on the regularity of $\theta$.)

\vspace{2mm} Within these results we can deduce that
\begin{equation*}
\dim H^*_c(\Sigma)_{\theta^{-1},\theta}=\sum_{v\in \mathrm{Stab}_{W(T)^F}(\theta)}1=\#\mathrm{Stab}_{W(T)^F}(\theta),
\end{equation*}
which concludes the theorem. The remaining part of this section devotes to the proofs of these lemmas.

\vspace{2mm} We start with the first two, which are much easier and can be proved simultaneously:

\begin{proof}[Proof of Lemma~\ref{lemm a} and Lemma~\ref{lemm b}]
First, by taking the change of variables $xF(u'u^{-})\mapsto x$, we can rewrite $\Sigma''_{v,\beta,i}$ as 
\begin{equation*}
\widetilde{\Sigma}''_{v,\beta,i} :=\{(x,x',u',u^{-},\tau,u)\in F(U^{r-b,b})\times FU\times \mathcal{Z}_v^{\beta}(i)''\times T\times U\mid xF(\hat{v}\tau)=u'u^-\hat{v}\tau ux'\},
\end{equation*}
on which the $T^F\times T^F$-action does not change. 

\vspace{2mm} Consider $H=\{(t,t')\in T_1\times T_1\mid tF(t^{-1})=F(\hat{v})t'F(t')^{-1}F(\hat{v}^{-1})\}$; this is an algebraic group and it acts on both $\Sigma_{v,c}$ and $\widetilde{\Sigma}''_{v,\beta,i}$ by naturally extending the $T_1^F\times T_1^F$-action (note that $T_1$ is a subgroup of $T$). The identity component $H^{\circ}$ is a torus, thus by basic properties of $\ell$-adic cohomology (see e.g.\ \cite[10.15]{DM1991}) we have
\begin{equation*}
\dim H^*_c(\Sigma_{v,c})_{\theta^{-1},\theta}=\dim H^*_c((\Sigma_{v,c})^{H^{\circ}})_{\theta^{-1},\theta}
\end{equation*}
and
\begin{equation*}
\dim H^*_c(\widetilde{\Sigma}''_{v,\beta,i})_{\theta^{-1},\theta}=\dim H^*_c((\widetilde{\Sigma}''_{v,\beta,i})^{H^{\circ}})_{\theta^{-1},\theta}.
\end{equation*}
The Lang--Steinberg theorem implies that both the first and the second projections of $H^{\circ}$ to $T_1$ are surjective, thus
\begin{equation*}
(\Sigma_{v,c})^{H^{\circ}}=\{ (1,1,1,1,\tau,1)\mid F(\hat{v}\tau)=\hat{v}\tau \}^{H^{\circ}}\quad \text{and}\quad (\widetilde{\Sigma}''_{v,\beta,i})^{H^{\circ}}=\emptyset.
\end{equation*}
So it remains to deal with $(\Sigma_{v,c})^{H^{\circ}}$.

\vspace{2mm} Note that $(\hat{v}T)^F$ is non-empty only if $v\in W(T)^F$; we only need to deal with this non-empty case. As $\{ (1,1,1,1,\tau,1)\mid F(\hat{v}\tau)=\hat{v}\tau \}$ is a finite set, it admits only the trivial action of the connected group $H^{\circ}$, thus 
\begin{equation*}
(\Sigma_{v,c})^{H^{\circ}}=(\hat{v}T)^F,
\end{equation*}
hence $H^*_c(\Sigma_{v,c})=\overline{\mathbb{Q}}_{\ell}[(\hat{v}T)^F]$, on which $T^F\times T^F$ acts by $(t,t')\colon \hat{v}\tau\mapsto \hat{v}(t^{\hat{v}})^{-1}\tau t'$ (note that this is the regular representation of both the left $T^F$ and the right $T^F$ in $T^F\times T^F$). In particular, the irreducible subrepresentations of $H^*_c(\Sigma_{v,c})$ are of the form $(\phi^{\hat{v}})^{-1}\times\phi$, where $\phi$ runs over $\widehat{T^F}$. Therefore, $H^*_c(\Sigma_{v,c})_{\theta^{-1},\theta}$ is non-zero (in which case it is of dimension $1$) if and only if $v\in W(T)^F$ and $\theta^{\hat{v}}=\theta$, that is, if and only if $v\in \mathrm{Stab}_{W(T)^F}(\theta)$.
\end{proof}

\vspace{2mm} Now we turn to Lemma~\ref{lemm c}; its proof is more difficult than those of Lemma~\ref{lemm a} and Lemma~\ref{lemm b}. We divide it into three cases, namely 
\begin{itemize}
\item[(I)]\ $i\geq b$;

\item[(II)]\ $i<r-b$;

\item[(III)]\ $r-b\leq i<b$.
\end{itemize}
We will treat them separately.

\begin{proof}[Proof of Lemma~\ref{lemm c} (case (I) $i\geq b$).]
In this case, $\Sigma'_{v,\beta,i}=\emptyset$ by our construction.
\end{proof}

In order to deal with the other two cases, we need the following slight generalisation of the technical result \cite[Lemma~4.6]{ChenStasinski_2016_algebraisation}:
\begin{lemm}\label{technical lemma}
For $i\in\{0,\dots,b-1\}$, $(u',u^-)\in\mathcal{Z}_v^{\beta}(i)'$, $z:=(u'u^-)^{\hat{v}}$, and $\xi\in U_{-F(\beta)}^{r-i-1}$, one has
\begin{equation*}
[\xi,F(z)]:=\xi F(z)\xi^{-1}F(z)^{-1}=\tau_{\xi,F(z)}\omega_{\xi,F(z)},
\end{equation*}
where $\tau_{\xi,F(z)}\in \mathcal{T}^{-F(\beta)}$ and $\omega_{\xi,F(z)}\in F(U^{-})^{r-1}$ are uniquely determined. Moreover,
\begin{equation*}
U_{-F(\beta)}^{r-i-1}\longrightarrow\mathcal{T}^{-F(\beta)},\qquad \xi\longmapsto\tau_{\xi,F(z)}
\end{equation*}
is a surjective morphism admitting a section $\Psi^{-F(\beta)}_{F(z)}$ such that $\Psi^{-F(\beta)}_{F(z)}(1)=1$ and such that the map
\begin{equation*}
\mathcal{Z}_v^{\beta}(i)'\times \mathcal{T}^{-F(\beta)}\longrightarrow U^{r-i-1}_{-F(\beta)},\qquad ((u',u^{-}),\tau)\longmapsto \Psi_{F(z)}^{-F(\beta)}(\tau)
\end{equation*}
is a morphism.
\end{lemm}
\begin{proof}
A similar argument of \cite[Lemma~4.6]{ChenStasinski_2016_algebraisation} works here (actually, it works for any $i\in\{0,...,r-2\}$); we record it for the completeness and for that we will use part of the argument later. Write $F(z)=x^{F(z)}_{F(\beta)}F(z')$ (see the notation in \eqref{formula: product wrt roots}), then 
\begin{equation}\label{temp1}
[\xi,Fz]=\xi \cdot x^{F(z)}_{F(\beta)}\cdot Fz'\cdot\xi^{-1}\cdot (Fz')^{-1}\cdot(x^{F(z)}_{F(\beta)})^{-1}=[\xi,x^{F(z)}_{F(\beta)}]\cdot{^{x^{F(z)}_{F(\beta)}}[\xi,F(z')]}.
\end{equation}
We need to determine $[\xi,x^{F(z)}_{F(\beta)}]$ and ${^{x^{F(z)}_{F(\beta)}}[\xi,Fz']}$.

\vspace{2mm} Following the notation in \cite[XX]{SGA3}, we write $p_{\alpha}\colon (\mathbb{G}_a)_{\mathcal{O}^{\mathrm{ur}}_r}\cong\mathbf{U}_{\alpha}$ for every $\alpha\in\Phi$ (and we use the same notation $p_{\alpha}$ for the isomorphism $\mathcal{F}(\mathbb{G}_a)_{\mathcal{O}^{\mathrm{ur}}_r}\cong U_{\alpha}$ induced by the Greenberg functor). Then there exists $a\in\mathbb{G}_m(\mathcal{O}^{\mathrm{ur}}_r)$ such that, for all $x,y\in \mathbb{G}_a(\mathcal{O}^{\mathrm{ur}}_r)$, we have
\begin{equation}\label{product of opposite root subgroup}
p_{-\alpha}(y)p_{\alpha}(x)=p_{\alpha}(\frac{x}{1+axy})\check{\alpha}((1+axy)^{-1})p_{-\alpha}(\frac{y}{1+axy}).
\end{equation}
(See \cite[XX 2.2]{SGA3}.) Let $x$ and $y$ be such that $p_{-F(\beta)}(x)=\xi$ and $p_{F(\beta)}(y)=x^{Fz}_{F(\beta)}$ (note that in our case $(xy)^2=0$, so $(1+axy)^{-1}=1-axy$). By applying (\ref{product of opposite root subgroup}) to the commutator $[p_{\alpha}(x),p_{-\alpha}(y)]=p_{\alpha}(x)p_{-\alpha}(y)p_{\alpha}(-x)p_{-\alpha}(-y)$ with $\alpha:=-F(\beta)$, we see that
\begin{equation}\label{temp2}
\begin{split}
[\xi,x^{Fz}_{F\beta}]&=p_{\alpha}(x)p_{-\alpha}(y)p_{\alpha}(-x)p_{-\alpha}(-y)\\
&=p_{\alpha}(x)p_{\alpha}(\frac{-x}{1-axy})\check{\alpha}(1+axy)p_{-\alpha}(\frac{y}{1-axy})p_{-\alpha}(-y)\\
&=\check{\alpha}(1+axy)p_{-\alpha}(axy^2).
\end{split}
\end{equation}
Note that, since $\xi\in G^{r-i-1}$ and $x_{F(\beta)}^{F(z)}\in G^i$ (in other words, $\pi^{r-i-1}\mid x$ and $\pi^i\mid y$), we have $p_{F(\beta)}(axy^2)\in U_{F(\beta)}^{r-1}$ (note that $p_{F(\beta)}(axy^2)=1$ unless $i=0$). In the below one shall see that $\check{\alpha}(1+axy)$ is the desired $\tau_{\xi,z}$.

\vspace{2mm} Now turn to $[\xi,F(z')]$; we want to show that $[\xi,F(z')]\in F(U^{-})^{r-1}$. Let us do this by induction on $\#\{\beta'\in F\Phi^{-}\mid x^{Fz'}_{\beta'}\neq1\}$ (a lightly different argument is applied to the corresponding result in the published version of \cite{ChenStasinski_2016_algebraisation}).

\vspace{2mm} If $\#\{\beta'\in F\Phi^{-}\mid x^{Fz'}_{\beta'}\neq1\}=1$, then we have $F(z')=p_{\beta'}(y_0)$ for some $\beta'\in F(\Phi^{-})$ and $y_0\in \mathbb{G}_a(\mathcal{O}^{\mathrm{ur}}_r)$, so by the Chevalley commutator formula (see \cite[3.3.4.1]{Demazure_Summary_of_Thesis}) we get
\begin{equation*}
[\xi,Fz']=\prod_{j,j'\geq1,\  j\beta'+j'(-F(\beta))\in\Phi}p_{j\beta'+j'(-F(\beta))}(a_{j,j'}y_0^jx^{j'})\in\prod_{j,j'\geq1,\  j\beta'+j'(-F(\beta))\in\Phi}(U_{j\beta'+j'(-F\beta)})^{r-1}
\end{equation*}
for some $a_{j,j'}\in\mathcal{O}_r^{\mathrm{ur}}$. In this formula, if $j\beta'+j'(-F\beta)\in\Phi^{+}$ for some $j,j'$ and if $y_0^jx^{j'}$ is non-zero, then the non-zero coefficients of the simple roots in $\beta'+(-F\beta)$ are all greater than zero and there exists at least one non-zero coefficient (recall that $\xi\in G^{r-i-1}$ and $z'\in G^i$); this implies that $\mathrm{ht}(F(\beta))>\mathrm{ht}(\beta')$, which contradicts our assumption on $z$, so $[\xi,F(z')]\in F(U^-)^{r-1}$ in this case. 

\vspace{2mm} Now, suppose that $[\xi,Fz']\in F(U^-)^{r-1}$ for all $\#\{\beta'\in F\Phi^{-}\mid x^{Fz'}_{\beta}\neq1\}\leq N$. Then for $\#\{\beta'\in F\Phi^{-}\mid x^{Fz'}_{\beta'}\neq1\}=N+1$, take a decomposition $Fz'=\prod_{\beta'\in F\Phi^{-}}x^{Fz'}_{\beta'}=z'_1z'_2$ in a way such that both $[\xi,z'_1]$ and $[\xi,z'_2]$ are in $F(U^-)^{r-1}$; note that 
\begin{equation*}
[\xi,Fz']=[\xi,z'_1]\cdot{^{z'_1}[\xi,z'_2]}.
\end{equation*}
As $z'_1\in FU^{-}$, we have ${^{z'_1}[\xi,z'_2]}\in F(U^{-})^{r-1}$, thus $[\xi,Fz']\in F(U^-)^{r-1}$ also for $\#\{\beta'\in F\Phi^{-}\mid x^{Fz'}_{\beta'}\neq1\}=N+1$. So by the induction principle we always have $[\xi,Fz']\in F(U^-)^{r-1}$.

\vspace{2mm} By \eqref{temp1} and \eqref{temp2} we get
\begin{equation*}
[\xi,Fz]=[\xi,x^{Fz}_{F(\beta)}]\cdot{^{x^{Fz}_{F\beta}}[\xi,Fz']}=\check{\alpha}(1+axy)\cdot p_{F\beta}(axy^2)\cdot{^{x^{Fz}_{F\beta}}[\xi,Fz']}.
\end{equation*}
Now put
\begin{equation*}
\tau_{\xi,Fz} =\check{\alpha}(1+axy)
\end{equation*}
and 
\begin{equation*}
\omega_{\xi,Fz}=p_{F\beta}(axy^2)\cdot{^{x^{Fz}_{F\beta}}[\xi,Fz']}.
\end{equation*}
From the above we see that $\tau_{\xi,Fz}\in\mathcal{T}^{-F\beta}$ and $\omega_{\xi,Fz}\in F(U^-)^{r-1}$ (as $[\xi,Fz']\in F(U^{-})^{r-1}$). The elements $\tau_{\xi,Fz}$ and $\omega_{\xi,Fz}$ are uniquely determined because of the Iwahori decomposition.

\vspace{2mm} Now, as $\tau_{\xi,F(z)}$ is defined to be $\check{\alpha}(1+ap^{-1}_{-F(\beta)}(\xi)p^{-1}_{F\beta}(x^{Fz}_{F\beta}))$, the map $\xi\mapsto\tau_{\xi,Fz}$, whose target is a connected $1$-dimensional algebraic group, is a surjective algebraic group morphism (note that $Fz\mapsto x_{F\beta}^{Fz}$ is a projection, hence a morphism). We define $\Psi^{-F\beta}_{Fz}$ in the following way: The isomorphism of additive groups 
\begin{equation*}
(\pi^i)\cong \mathcal{O}^{\mathrm{ur}}_{r-i},\quad \pi^ia+(\pi^r)\longmapsto a+(\pi^{r-i})
\end{equation*}
induces an isomorphism of affine spaces (by the Greenberg functor)
\begin{equation*}
\mu_i\colon (\mathcal{F}(\mathbb{G}_{a})_{\mathcal{O}^{\mathrm{ur}}_r})^i\longrightarrow(\mathcal{F}(\mathbb{G}_{a})_{\mathcal{O}^{\mathrm{ur}}_r})_{r-i}.
\end{equation*}
Note that this isomorphism depends on the choice of $\pi$. Meanwhile, let
\begin{equation*}
\mu^i\colon (\mathcal{F}(\mathbb{G}_{a})_{\mathcal{O}^{\mathrm{ur}}_r})_{r-i}\cong\mathcal{F}(\mathbb{G}_{a})_{\mathcal{O}^{\mathrm{ur}}_r}/(\mathcal{F}(\mathbb{G}_{a})_{\mathcal{O}^{\mathrm{ur}}_r})^{r-i}\longrightarrow\mathcal{F}(\mathbb{G}_{a})_{\mathcal{O}^{\mathrm{ur}}_r}
\end{equation*}
be a section morphism to the quotient morphism such that $\mu^i(0)=0$ ($\mu^i$ exists because $\mathcal{F}(\mathbb{G}_{a})_{\mathcal{O}^{\mathrm{ur}}_r}$ is an affine space). For $\tau\in\mathcal{T}^{-F\beta}$ we put
\begin{equation*}
\Psi_{Fz}^{-F\beta}(\tau):=p_{-F\beta}\left(a^{-1}\cdot\mu^{i}\left(  \mu_i\left( \check{\alpha}^{-1}(\tau)-1  \right)\cdot \mu_i\left(  p_{F\beta}^{-1}(x^{Fz}_{F\beta}) \right)^{-1}   \right) \right).
\end{equation*}
Here $\check{\alpha}^{-1}$ is defined on $\mathcal{T}^{-F\beta}=(\mathcal{F}\mathbf{T}^{-F\beta})^{r-1}\cong(\mathcal{F}(\mathbb{G}_{m})_{\mathcal{O}^{\mathrm{ur}}_r})^{r-1}$ as the inverse to $\check{\alpha}$, and we view $\check{\alpha}^{-1}(\tau)$ as an element in $\mathcal{F}(\mathbb{G}_{a})_{\mathcal{O}^{\mathrm{ur}}_r}$ by the natural open immersion $(\mathbb{G}_{m})_{\mathcal{O}^{\mathrm{ur}}_r}\rightarrow(\mathbb{G}_{a})_{\mathcal{O}^{\mathrm{ur}}_r}$, so the minus operation $ \check{\alpha}^{-1}(\tau)-1$ is well-defined. On the other hand, by our assumption on $z$, $\mu_i\left( p_{F\beta}^{-1}(x^{Fz}_{F\beta}) \right)$ is an element in $\mathcal{F}(\mathbb{G}_{m})_{\mathcal{O}^{\mathrm{ur}}_{r-i}}$, so its multiplicative inverse exists. Moreover, the product operation ``$\cdot$'' is by viewing $(\mathbb{G}_{a})_{\mathcal{O}^{\mathrm{ur}}_r}$ (resp.\ $\mathcal{F}(\mathbb{G}_{a})_{\mathcal{O}^{\mathrm{ur}}_r}$) as a ring scheme (resp.\ $k$-ring variety). Thus $\Psi_{F(z)}^{-F(\beta)}$ is well-defined; we need to check that it is a section morphism.

\vspace{2mm} By the definition of $\mu_i$ and $\mu^i$, for $\tau\in\mathcal{T}^{-F\beta}(k)$ we have
\begin{equation*}
\begin{split}
\tau_{\Psi_{Fz}^{-F\beta}(\tau),Fz}&=\check{\alpha}\left(1+ap_{F\beta}^{-1}(\Psi_{Fz}^{F\beta}(\tau))p^{-1}_{F\beta}(x_{F\beta}^{Fz})\right)\\
&=\check{\alpha}\left(1+\mu^{i}\left(  \mu_i\left(  \check{\alpha}^{-1}(\tau)-1  \right)\cdot \mu_i\left(  p_{F\beta}^{-1}(x^{Fz}_{F\beta}) \right)^{-1}  \right) \cdot p^{-1}_{F\beta}(x_{\beta}^{Fz}) \right)\\
&=\check{\alpha}\left(1+\pi^i\cdot\mu^i\mu_i(\check{\alpha}^{-1}(\tau)-1)\right)=\tau.
\end{split}
\end{equation*}
(For the third line, note that $\check{\alpha}^{-1}(\tau)$ is of the form $1+s\pi^{r-1}$ for some $s\in\mathcal{O}^{\mathrm{ur}}_r$, as an element in $\mathbb{G}_m(\mathcal{O}^{\mathrm{ur}}_r)$.) Thus $\tau\mapsto\Psi_{Fz}^{-F\beta}(\tau)\mapsto\tau_{\Psi_{Fz}^{-F\beta}(\tau),Fz}$ is the identity map on the $k$-points $\mathcal{T}^{-F\beta}(k)$ of the $1$-dimensional affine space $\mathcal{T}^{-F\beta}\cong\mathbb{A}^1_k$, hence it is the identity morphism, so $\Psi_{Fz}^{-F\beta}$ is a section to $\xi\mapsto\tau_{\xi,F(z)}$; the other assertions follow from the definition of $\Psi_{Fz}^{-F\beta}$.
\end{proof}

Taking the changes of variables $\hat{v}\tau\hat{v}^{-1}\mapsto\tau$, $\tau^{-1}u^{-}\tau\mapsto u^{-}$, and then $\tau^{-1}u'\tau\mapsto u'$ (in this order), we can rewrite $\Sigma'_{v,\beta,i}$ as
\begin{equation*}
\widetilde{\Sigma}'_{v,\beta,i}:=\{(x,x',u',u^{-},\tau,u)\in FU^{r-b,b}\times FU\times \mathcal{Z}_v^{\beta}(i)'\times T\times U\mid xF(\tau u'u^{-}\hat{v})=\tau u'u^{-}\hat{v}ux'\},
\end{equation*}
on which $(t,t')\in T^F\times T^F$ acts by sending $(x,x',u',u^{-},\tau,u)$ to
\begin{equation*} 
(t^{-1}xt,{t'}^{-1}x't',({^{\hat{v}}t'})^{-1}u'(^{\hat{v}}t'),({^{\hat{v}}t'})^{-1}u^-(^{\hat{v}}t'),t^{-1}\tau (^{\hat{v}}t'),{t'}^{-1}ut').
\end{equation*}
To show $\dim H^*_c(\widetilde{\Sigma}'_{v,\beta,i})_{\theta^{-1},\theta}=0$ for $i<b$, it suffices to show 
\begin{equation*}
\dim H^*_c(\widetilde{\Sigma}'_{v,\beta,i})_{\theta^{-1}|_{(T^{r-1})^F}}=0
\end{equation*}
for $i<b$, where the subscript $(T^{r-1})^F$ denotes the subgroup $(T^{r-1})^F\times 1\subseteq T^F\times T^F$. Note that $(T^{r-1})^F$ acts on $\widetilde{\Sigma}'_{v,\beta,i}$ by 
\begin{equation*}
t\colon (x,x',u',u^{-},\tau,u)\mapsto (x,x',u',u^{-},t^{-1}\tau,u).
\end{equation*}

\begin{proof}[Proof of Lemma~\ref{lemm c} (case (II) $i<r-b$).]
In this case, note that $\mathcal{Z}_{v}^{\beta}(i)'\subseteq\mathcal{Z}_v$ is just the closed subvariety consisting of the pairs $(u',u^{-})$ satisfying that: 
\begin{itemize}
\item[(1)] $z:=(u'u^{-})^{\hat{v}}\in (U^-)^i\backslash (U^-)^{i+1}$;

\item[(2)] $x_{F(\beta)}^{F(z)}\neq1$;

\item[(3)] $x_{\beta'}^{F(z)}=1$ for all $\beta'<F(\beta)$.
\end{itemize}

\vspace{2mm} Consider
\begin{equation*}
H_{\beta}:=\{t\in T^{r-1}\mid F(\hat{v})^{-1}F(t)t^{-1}F(\hat{v})\in \mathcal{T}^{-F(\beta)}\};
\end{equation*}
this is a closed subgroup of $T^{r-1}$, and it contains $(T^{r-1})^F$.

\vspace{2mm} For any $t\in H_{\beta}$, we have a morphism $g_t\colon FU\rightarrow FU$ given by
\begin{equation*}
g_t\colon x'\mapsto x'\cdot\Psi_{F(z)}^{-F(\beta)}\left(F(\hat{v})^{-1}F(t^{-1})tF(\hat{v})\right)^{-1},
\end{equation*}
where the parameter is $z:=\hat{v}^{-1}u'u^{-}\hat{v}$ with $(u',u^{-})\in \mathcal{Z}_{v}^{\beta}(i)'$ (here $\Psi_{F(z)}^{-F(\beta)}$ is the morphism in Lemma~\ref{technical lemma}). By Lemma~\ref{technical lemma}, if $F(t)=t$, then $g_t=\mathrm{Id}$.

\vspace{2mm} Meanwhile, for any $t\in H_{\beta}$, we have a morphism $f_t\colon FU^{r-b,b}\rightarrow FU^{r-b,b}$ given by
\begin{equation*}
f_t\colon x\mapsto x\cdot {^{F(\tau)}\left( t^{-1} \cdot {^{F(\hat{v}z)}\left( x'^{-1}g_t(x')\right)}  F(t) \right)},
\end{equation*}
where the parameters are $x'\in FU$, $\tau\in T$, and $z:=\hat{v}^{-1}u'u^{-}\hat{v}$ with $(u',u^{-})\in\mathcal{Z}_v^{\beta}(i)'$. We need to check that this is well-defined, i.e.\ to check that the right hand side is in $FU^{r-b,b}$: By Lemma~\ref{technical lemma} we have
\begin{equation*}
F(z) x'^{-1}g_t(x') F(z^{-1})=\Psi_{F(z)}^{-F(\beta)}\left(F(\hat{v})^{-1}F(t^{-1})tF(\hat{v})\right)^{-1}\cdot F(\hat{v})^{-1}F(t^{-1})tF(\hat{v})\cdot\omega
\end{equation*}
for some $\omega\in{U^{r-1}(U^-)^{r-1}}$, so
\begin{equation*}
(x^{-1}f_t(x))^{F(\tau)}=({^{F(\hat{v})}\Psi})^t   \cdot ({^{F(\hat{v})}\omega})^{F(t)} \in\prod_{\alpha\in\Phi}{U^{r-i-1}_{\alpha}}\subseteq \prod_{\alpha\in\Phi} (U_{\alpha})^b\subseteq FU^{r-b,b},
\end{equation*}
where $\Psi:=\Psi_{F(z)}^{-F(\beta)}(F(\hat{v})^{-1}F(t^{-1})tF(\hat{v}))^{-1}$. Thus $x^{-1}f_t(x)\in FU^{r-b,b}$, hence $f_t$ is well-defined. Like $g_t$, if $F(t)=t$, then $f_t=\mathrm{Id}$.

\vspace{2mm} Now for any $t\in H_{\beta}$, by combining the above constructions we get the following automorphism of $\widetilde{\Sigma}'_{v,\beta,i}$:
\begin{equation*}
h_t\colon (x,x',u',u^{-},\tau,u)\mapsto (f_t(x),g_t(x'),u',u^{-},t^{-1}\tau,u),
\end{equation*}
where the parameter is $z:=\hat{v}^{-1}u'u^{-}\hat{v}$. We need to check that this is well-defined, i.e.\ to check that the image satisfies the defining equation of $\widetilde{\Sigma}'_{v,\beta,i}$, that is, satisfies 
\begin{equation*}
f_t(x)F(t^{-1}\tau u'u^{-}\hat{v})=t^{-1}\tau u'u^{-}\hat{v}ug_t(x').
\end{equation*}
Indeed, by expanding the definition of $f_t$ we see that: 
\begin{equation*}
\begin{split}
f_t(x)F(t^{-1}\tau u'u^{-}\hat{v})&=x\cdot {^{F(\tau)}\left( t^{-1} \cdot {^{F(\hat{v}z)}\left( x'^{-1}g_t(x')\right)} F(t) \right)}\cdot F(t^{-1}\tau u'u^{-}\hat{v})\\
&=t^{-1}xF(\tau u'u^{-}\hat{v})x'^{-1}g_t(x')\\
&=t^{-1}\tau u'u^{-}\hat{v}ug_t(x').
\end{split}
\end{equation*}
(For the second line, note that $t\in T^{r-1}$ commutes with the elements in $G^{r-1}$; for the third line, use the property $xF(\tau u'u^{-}\hat{v})=\tau u'u^{-}\hat{v}ux'$.) So $h_t$ is well-defined. Clearly, if $F(t)=t$, then $h_t$ coincides with the $(T^{r-1})^F$-action, so (see the discussions in \cite[p.~136]{DL1976} or \cite[Lemma~4.3.4]{ZheChen_PhDthesis}) the induced endomorphism of $h_t$ on $H^*_c(\widetilde{\Sigma}'_{v,\beta,i})$ is the identity map for any $t$ in the identity component $(H_{\beta})^{\circ}$ of $H_{\beta}$.

\vspace{2mm} Let $a\in\mathbb{Z}_{>0}$ be such that $F^a(F(\hat{v})\mathcal{T}^{-F(\beta)}F(\hat{v})^{-1})=F(\hat{v})\mathcal{T}^{-F(\beta)}F(\hat{v})^{-1}$. By continuity, the images of the norm map $N^{F^a}_F(t)=t\cdot F(t)\cdots F^{a-1}(t)$ on $F(\hat{v})\mathcal{T}^{-F(\beta)}F(\hat{v})^{-1}$ form a connected subgroup of $H_{\beta}$, hence are contained in $(H_{\beta})^{\circ}$. Therefore $N^{F^a}_F((F(\hat{v})\mathcal{T}^{-F(\beta)}F(\hat{v})^{-1})^{F^a})\subseteq (T^{r-1})^F\cap (H_{\beta})^{\circ}$. Thus, the regularity of $\theta$ implies
\begin{equation*}
H^*_c(\widetilde{\Sigma}'_{v,\beta,i})_{\theta^{-1}\big|_{ N^{F^a}_F\left(\left(F(\hat{v})\mathcal{T}^{-F(\beta)}F(\hat{v})^{-1}\right)^{F^a}\right)}}=0.
\end{equation*}
In particular, $H^*_c(\widetilde{\Sigma}'_{v,\beta,i})_{\theta^{-1}|_{(T^{r-1})^F}}=0$.
\end{proof}

We use a modification of the argument of the case (II) to deal with the case (III).

\begin{proof}[Proof of Lemma~\ref{lemm c} (case (III) $r-b\leq i<b$).]
In this case, by our construction, $\widetilde{\Sigma}'_{v,\beta,i}$ (or more precisely, $\mathcal{Z}_{v}^{\beta}(i)'$) is non-empty only if there is a $\beta_1\in\Phi^-$ such that $v(\beta_1)\in\Phi^{-}$; we only need to deal with this non-empty case. Denote by $\mathcal{B}_v$ the set consisting of the elements $\beta_1 \in \Phi^-$ satisfying: $F(\beta_1)\geq F(\beta)$ and $v(\beta_1)\in\Phi^-$. Then by the assumption that $r-b\leq i<b$ we have a stratification by locally closed subvarieties
\begin{equation*}
\mathcal{Z}_{v}^{\beta}(i)'=\sqcup_{\beta_1\in\mathcal{B}_v}\mathcal{Z}_{v}^{\beta,\beta_1}(i),
\end{equation*}
where $\mathcal{Z}_{v}^{\beta,\beta_1}(i)\subseteq \mathcal{Z}_{v}^{\beta}(i)'$ is the subvariety consisting of the pairs $(u',u^-)\in \mathcal{Z}_{v}^{\beta}(i)'$ satisfying the following properties: (write $z:=(u'u^{-})^{\hat{v}}$ for $(u',u^-)\in\mathcal{Z}_{v}^{\beta}(i)'$)
\begin{itemize}
\item[(1)] $x^{F(z)}_{F(v(\beta_1))}\notin F(U^-)^{b}$;

\item[(2)] $x^{F(z)}_{F(v(\beta_2))}\in F(U^-)^{b}$ for all $\beta_2\in\mathcal{B}_v$ satisfying that $F(\beta_2)<F(\beta_1)$.
\end{itemize}
This partition of $\mathcal{Z}_{v}^{\beta}(i)'$ naturally induces a partition of $\widetilde{\Sigma}'_{v,\beta,i}$:
\begin{equation*}
\widetilde{\Sigma}'_{v,\beta,i}=\sqcup_{\beta_1\in\mathcal{B}_v}{\Sigma}_{v,\beta,\beta_1,i},
\end{equation*}
where
\begin{equation*}
{\Sigma}_{v,\beta,\beta_1,i}:=\{(x,x',u',u^{-},\tau,u)\in FU^{r-b,b}\times FU\times \mathcal{Z}_{v}^{\beta,\beta_1}(i)\times T\times U\mid xF(\tau u'u^{-}\hat{v})=\tau u'u^{-}\hat{v}ux'\};
\end{equation*}
clearly each ${\Sigma}_{v,\beta,\beta_1,i}$ inherits the $(T^{r-1})^F$-action.

\vspace{2mm} Consider
\begin{equation*}
H_{\beta_1}:=\{t\in T^{r-1}\mid F(\hat{v})^{-1}F(t)t^{-1}F(\hat{v})\in \mathcal{T}^{-F(\beta_1)}\},
\end{equation*}
which is a closed subgroup of $T^{r-1}$ containing $(T^{r-1})^F$. In the following, for $(u',u^-)\in \mathcal{Z}_{v}^{\beta,\beta_1}(i)$, we put $z:=(u'u^-)^{\hat{v}}$ and we write $F(z)=F(z_0)\cdot F(z_1)$, where 
$$F(z_0):=\prod_{\alpha\in F(\Phi^{-}),\ \beta'<F(\beta_1)}x^{F(z)}_{\beta'}$$
and
$$F(z_1):=\prod_{\alpha\in F(\Phi^{-}),\ \beta'\geq F(\beta_1)}x^{F(z)}_{\beta'};$$
here the products are taken in the order as in \eqref{formula: product wrt roots}.

\vspace{2mm} For $t\in H_{\beta_1}$, we have a morphism $g_t\colon FU\rightarrow FU$ defined by
\begin{equation*}
g_t\colon x'\mapsto x'\cdot\Psi_{F(z_1)}^{-F(\beta_1)}\left(F(\hat{v})^{-1}F(t^{-1})tF(\hat{v})\right)^{-1},
\end{equation*}
where the parameter is $z:=\hat{v}^{-1}u'u^{-}\hat{v}$ with $(u',u^{-})\in \mathcal{Z}_{v}^{\beta,\beta_1}(i)$; this $\Psi_{F(z_1)}^{-F(\beta_1)}$ is the morphism in Lemma~\ref{technical lemma}. Note that, if $F(t)=t$, then $g_t=\mathrm{Id}$.

\vspace{2mm} Meanwhile, for $t\in H_{\beta_1}$, we have a morphism $f_t\colon FU^{r-b,b}\rightarrow FU^{r-b,b}$ defined by
\begin{equation*}
f_t\colon x\mapsto x\cdot {^{F(\tau)}\left( t^{-1} \cdot {^{F(\hat{v}z)}\left( x'^{-1}g_t(x')\right)}\cdot F(t) \right)},
\end{equation*}
where the parameters are $x'\in FU$, $\tau\in T$, and $z:=\hat{v}^{-1}u'u^{-}\hat{v}$ with $(u',u^{-})\in \mathcal{Z}_{v}^{\beta,\beta_1}(i)$. To see this is well-defined, we need to check that the right hand side is in $FU^{r-b,b}$: By Lemma~\ref{technical lemma} we have
\begin{equation*}
F(z) x'^{-1}g_t(x') F(z^{-1})={^{F(z_0)}(\Psi_{F(z_1)}^{-F(\beta_1)}\left(F(\hat{v})^{-1}F(t^{-1})tF(\hat{v})\right)^{-1}\cdot F(\hat{v})^{-1}F(t^{-1})tF(\hat{v})\cdot\omega)}
\end{equation*}
for some $\omega\in{U^{r-1}(U^-)^{r-1}}$, so (recall that $i>0$)
\begin{equation*}
(x^{-1}f_t(x))^{F(\tau)}=({^{F(\hat{v}z_0)}\Psi})^t   \cdot ({^{F(\hat{v})}\omega})^{F(t)}
\end{equation*}
where $\Psi:=\Psi_{F(z_1)}^{-F(\beta_1)}(F(\hat{v})^{-1}F(t^{-1})tF(\hat{v}))^{-1}$; it remains to show that ${^{F(\hat{v}z_0)}\Psi}\in FU^{r-b,b}$. By the Chevalley commutator formula (see \cite[3.3.4.1]{Demazure_Summary_of_Thesis} or the relevant part in the proof of Lemma~\ref{technical lemma}) we have ${^{F(z_0)}\Psi}=\Psi\cdot\omega_1$ for some $\omega_1\in U^{r-1}(U^-)^{r-1}$. Now, as $b>0$, it suffices to show that ${^{F(\hat{v})}\Psi}\in FU^{r-b,b}$, which follows immediately from our assumption that $\beta_1\in\mathcal{B}_v$. Thus $f_t$ is well-defined. And like $g_t$, if $F(t)=t$, then $f_t=\mathrm{Id}$.

\vspace{2mm} Now for any $t\in H_{\beta_1}$, by combining the above constructions we get the following automorphism of ${\Sigma}_{v,\beta,\beta_1,i}$:
\begin{equation*}
h_t\colon (x,x',u',u^{-},\tau,u)\mapsto (f_t(x),g_t(x'),u',u^{-},t^{-1}\tau,u),
\end{equation*}
where the parameter is $z:=\hat{v}^{-1}u'u^{-}\hat{v}$. As in the case (II), we can check that this is well-defined by expanding the definition of $f_t$:
\begin{equation*}
\begin{split}
f_t(x)F(t^{-1}\tau u'u^{-}\hat{v})&=x\cdot {^{F(\tau)}\left( t^{-1} \cdot {^{F(\hat{v}z)}\left( x'^{-1}g_t(x')\right)} F(t) \right)}\cdot F(t^{-1}\tau u'u^{-}\hat{v})\\
&=t^{-1}xF(\tau u'u^{-}\hat{v})x'^{-1}g_t(x')\\
&=t^{-1}\tau u'u^{-}\hat{v}ug_t(x').
\end{split}
\end{equation*}
So $h_t$ is well-defined. Similarly, if $F(t)=t$, then $h_t$ coincides with the $(T^{r-1})^F$-action, so the induced endomorphism of $h_t$ on $H^*_c({\Sigma}_{v,\beta,\beta_1,i})$ is the identity map for any $t$ in the identity component $(H_{\beta_1})^{\circ}$ of $H_{\beta_1}$.

\vspace{2mm} Let $a\in\mathbb{Z}_{>0}$ be such that $F^a(F(\hat{v})\mathcal{T}^{-F(\beta_1)}F(\hat{v})^{-1})=F(\hat{v})\mathcal{T}^{-F(\beta_1)}F(\hat{v})^{-1}$. Again, the images of the norm map $N^{F^a}_F(t)=t\cdot F(t)\cdots F^{a-1}(t)$ on $F(\hat{v})\mathcal{T}^{-F(\beta_1)}F(\hat{v})^{-1}$ form a connected subgroup of $H_{\beta_1}$, hence are contained in $(H_{\beta_1})^{\circ}$. Thus $N^{F^a}_F((F(\hat{v})\mathcal{T}^{-F(\beta_1)}F(\hat{v})^{-1})^{F^a})\subseteq (T^{r-1})^F\cap (H_{\beta_1})^{\circ}$. Finally, the regularity of $\theta$ implies that
\begin{equation*}
H^*_c({\Sigma}_{v,\beta,\beta_1,i})_{\theta^{-1}\big|_{ N^{F^a}_F\left(\left(F(\hat{v})\mathcal{T}^{-F(\beta_1)}F(\hat{v})^{-1}\right)^{F^a}\right)}}=0.
\end{equation*}
Therefore $H^*_c({\Sigma}_{v,\beta,\beta_1,i})_{\theta^{-1}|_{(T^{r-1})^F}}=0$ for all $\beta_1\in\mathcal{B}_v$, so $H^*_c(\widetilde{\Sigma}'_{v,\beta,i})_{\theta^{-1}|_{(T^{r-1})^F}}=0$.
\end{proof}

This completes the proof of the theorem.

\bibliographystyle{alpha}
\bibliography{zchenrefs}

\end{document}